\documentclass[12pt]{article}
\usepackage{amssymb}
\usepackage{amsfonts}
\usepackage{amsmath}
\usepackage{amsthm}

\newtheorem{theorem}{Theorem}[section]
\newtheorem{corollary}[theorem]{Corollary}
\newtheorem{lemma}[theorem]{Lemma}

\theoremstyle{definition}

\newtheorem{conjecture}[theorem]{Conjecture}

\newtheorem{definition}[theorem]{Definition}

\newtheorem{remark}[theorem]{Remark}

\input{tcilatex}
\begin{document}

\title{Homology equivalences of manifolds and zero-in-the-spectrum examples}
\author{Shengkui Ye}
\maketitle
\begin{abstract}   
Working with group homomorphisms, a construction of manifolds is introduced
to preserve homology groups. The construction gives as special cases
Qullien's plus construction with handles obtained by Hausmann, the existence
of one-sided $h$-cobordism of Guilbault and Tinsley, the existence of
homology spheres and higher-dimensional knots proved by Karvaire. We also
use it to get counter-examples to the zero-in-the-spectrum conjecture found
by Farber-Weinberger, and by Higson-Roe-Schick.
\end{abstract}

\section{Introduction}

The aim of this note is to propose a general surgery plus construction on
manifolds. This is a manifold version of the generalized plus construction
for CW complexes found by the author in \cite{Ye}. As applications, we give
a unified approach to the plus construction with handles of Hausmann \cite%
{ha}, the (mod $L$)-one-sided $h$-cobordism of Guilbault and Tinsley \cite{gt},
the existence of homology spheres of Kervaire \cite{kar} and the existence
of higher-dimensional knots of Kervaire \cite{ker}. We also use it to get
some examples for the zero-in-the-spectrum conjecture found by
Farber-Weinberger \cite{FW} and Higson-Roe-Schick \cite{hig}. First, we
briefly review these existing works.

Let $M$ be an $n$-dimensional $(n\geq 5)$ closed manifold with fundamental
group $\pi _{1}(M)=H$. Suppose that $\Phi :H\rightarrow G$ is a surjective
group homomorphism of finitely presented groups with the kernel $\ker \Phi $
a perfect subgroup. Hausmann shows that Quillen's plus construction with
respect to $\ker \Phi $ can by made by adding finitely many two and three
handles to $M\times 1\subset M\times \lbrack 0,1]$ (cf. Section 3 in Hausmann \cite%
{ha} and Definition of $\varphi _{1}$ on page 115 in Hausmann \cite{ha2}). The
resulting cobordism $(W;M,M^{\prime })$ has $W$ and $M^{\prime }$ the
homotopy type of the Quillen plus construction $M^{+}.$ In other words, the
fundamental group $\pi _{1}(M^{\prime })=\pi _{1}(W)=G$ and for any $\mathbb{%
Z[}G]$-module $N,$ the inclusion map $M\hookrightarrow W$ induces homology
isomorphisms $H_{\ast }(M;N)\cong H_{\ast }(W;N).$

Guilbault and Tinsley \cite{gt} obtain a generalized manifold plus
construction, as follows. Let $(W,Y)$ be a connected CW pair and $L$ a
normal subgroup of $\pi _{1}(Y).$ The inclusion $Y\hookrightarrow W$ is
called a (mod $L$)-homology equivalence if it induces an isomorphism of
fundamental groups and is a homology equivalence with coefficients $\mathbb{Z%
}[\pi _{1}(Y)/L].$ A compact cobordism $(W;X,Y)$ is a \emph{(mod }$L$\emph{%
)-one-sided }$h$\emph{-cobordism} if $X\hookrightarrow W$ induces a
surjection of fundamental groups and $Y\hookrightarrow W$ is a (mod $L$%
)-homology equivalence. Let $B$ be a closed $n$-manifold $(n\geq 5)$ and $%
\alpha :\pi _{1}(B)\rightarrow G$ a surjective homomorphism onto a finitely
presented group such that $\ker (\alpha )$ is strongly $L^{\prime }$%
-perfect, \textsl{i.e.} $\ker (\alpha )=[\ker (\alpha ),L^{\prime }]$ for
some group $L^{\prime },$ where $\ker (\alpha )\trianglelefteq L^{\prime
}\trianglelefteq \pi _{1}(B)$. For $L=L^{\prime }/\ker (\alpha ),$ Guilbault
and Tinsley \cite{gt} show that there exists a (mod $L$)-one-sided $h$%
-cobordism $(W;B,A)$ such that $\pi _{1}(W)=G$ and $\ker (\pi
_{1}(B)\rightarrow \pi _{1}(W))=\ker (\alpha ).$

An $n$-dimensional homology sphere is a closed manifold $M$ having the
homology groups of the $n$-sphere, \textsl{i.e.} $H_{\ast }(M)\cong H_{\ast
}(S^{n}).$ Let $\pi $ be a finitely presented group and $n\geq 5.$ Kervaire
shows that there exists an $n$-dimensional homology sphere $M$ with $\pi
_{1}(M)=\pi $ if and only if the homology groups satisfy $H_{1}(\pi ;\mathbb{%
Z})=H_{2}(\pi ;\mathbb{Z})=0.$

For an integer $n\geq 1,$ define an $n$-knot to be a differential imbedding $%
f:S^{n}\rightarrow S^{n+2}$ and the group of the $n$-knot $f$ to be $\pi
_{1}(S^{n+2}-f(S^{n})).$ Let $G$ be a finitely presentable group. The weight
$w(G)$ is the smallest integer $k$ such that there exists a set of $k$
elements $\alpha _{1},\alpha _{2},\ldots ,\alpha _{k}\in G$ whose normal
closure equals $G.$ Kervaire \cite{ker} shows that a finitely presented
group $G$ is the group of a $n$-knot $(n\geq 5)$ if and only if $H_{1}(G;%
\mathbb{Z)=Z}$, $H_{2}(G;\mathbb{Z})=0$ and the weight of $G$ is $1.$

The zero-in-the-spectrum conjecture goes back to Gromov, who asked for a
closed, aspherical, connected and oriented Riemannian manifold $M$ whether
there always exists some $p\geq 0,$ such that zero belongs to the spectrum
of the Laplace-Beltrami operator $\Delta _{p}$ acting on the square
integrable $p$-forms on the universal covering $\tilde{M}$ of $M.$ Farber
and Weinberger \cite{FW} show that the conjecture is not true if the
condition that $M$ is aspherical is dropped. More generally, Higson, Roe and
Schick \cite{hig} show that for a finitely presented group $G$ satisfying $%
H_{0}(G;C_{r}^{\ast }(G))=H_{1}(G;C_{r}^{\ast }(G))=H_{2}(G;C_{r}^{\ast
}(G))=0,$ there always exists a closed manifold $Y$ of dimension $n$ $(n\geq
6)$ with $\pi _{1}(Y)=G$ such that $Y$ is a counterexample to the conjecture
if $M$ is not required to be aspherical.

In this note, a more general construction is provided to preserve homology
groups. For this, we have to introduce the notion of a finitely $G$-dense
ring.

\begin{definition}
\label{fgdense}A finitely $G$\emph{-dense ring} $(R,\phi )$ is a unital ring
$R$ together with a ring homomorphism $\phi :\mathbb{Z}[G]\rightarrow R$
such that, when $R$ is regarded as a left $\mathbb{Z}[G]$-module via $\phi ,$
for any finitely generated right $\mathbb{Z}[G]$-module $M,$ finitely
generated free right $R$-module $F$ and $R$-module surjection $%
f:M\bigotimes\nolimits_{\mathbb{Z}[G]}R\twoheadrightarrow F,$ the module $F$ has a
finite $R$-basis in $f(M\otimes 1).$
\end{definition}

This is an analog of $G$-dense rings defined in Ye \cite{Ye}. Examples of
finitely $G$-dense rings include the real reduced group $C^{\ast }$-algebra $%
C_{\mathbb{R}}^{\ast }(G),$ the real group von Neumann algebra $\mathcal{N}_{%
\mathbb{R}}G,$ the real Banach algebra $l_{\mathbb{R}}^{1}(G),$ the rings $k=%
\mathbb{Z}/p$ (prime $p$) and $k\subseteq \mathbb{Q}$ a subring of the
rationals (with trivial $G$-actions), the group ring $k[G],$ and so on.

\textit{Conventions.} Let $\pi $ and $G$ be two groups. Suppose that $R$ is
a $\mathbb{Z}[G]$-module and $BG,B\pi $ are the classifying spaces. For a
group homomorphism $\alpha :\pi \rightarrow G,$ we will denote by $%
H_{*}(G,\pi ;R)$ the relative homology group $H_{*}(BG,B\pi ;R)$ with
coefficients $R.$ All manifolds are assumed to be connected smooth
manifolds, until otherwise stated.

Our main result is the following.

\begin{theorem}
\label{th3}Assume that $G$ is a finitely presented group and $(R,\phi )$ is
a finitely $G$-dense ring. Let $X$ be a connected $n$-dimensional ($n\geq 5$%
) closed orientable manifold with fundamental group $\pi =\pi _{1}(X)$.
Assume that $\alpha :\pi \rightarrow G$ is a group homomorphism of finitely
presented groups such that the image $\alpha (\pi )$ is finitely presented
and
\begin{eqnarray*}
H_{1}(\alpha ) &:&H_{1}(\pi ;R)\rightarrow H_{1}(G;R)\text{ is injective,
and } \\
H_{2}(\alpha ) &:&H_{2}(\pi ;R)\rightarrow H_{2}(G;R)\text{ is surjective.}
\end{eqnarray*}%
Suppose either that $R$ is a principal ideal domain or that the relative
homology group $H_{1}(G,\pi ;R)$ is a stably free $R$-module. When $2$ is
not invertible in $R,$ suppose that the manifold $M$ is a spin manifold.
Then there exists a closed $R$-orientable manifold $Y$ with the following
properties:

\begin{enumerate}
\item[(i)] $Y$ is obtained from $X$ by attaching 1-handles, 2-handles and
3-handles, such that

\item[(ii)] $\pi _{1}(Y)\cong G$ and the inclusion map $g:X\rightarrow W,$
the cobordism between $X$ and $Y$, induces the same fundamental group
homomorphism as $\alpha ,$ and

\item[(iii)] for any integer $q\geq 2$ the map $g$ induces an isomorphism of
homology groups%
\begin{equation}
g_{\ast }:H_{q}(X;R)\overset{\cong }{\rightarrow }H_{q}(W;R).  \tag{1}
\end{equation}
\end{enumerate}
\end{theorem}

Theorem \ref{th3} has the following applications.\medskip

\begin{enumerate}
\item[(1)] \noindent $\alpha $ \textbf{surjective}.
\end{enumerate}

\begin{itemize}
\item When $\ker \alpha $ is perfect and $R=\mathbb{Z}[G]$ the group ring,
this is Quillen's plus construction with handles, which is obtained by
Hausmann \cite{ha} and \cite{ha2} (see Corollary \ref{plus}).

\item When $\ker \alpha $ is strongly $L^{\prime }$-perfect and $R=\mathbb{Z}%
[G/L],$ Theorem \ref{th3} implies the existence of (mod $L$)-one-sided $h$%
-cobordism obtained by Guilbault and Tinsley \cite{gt2, gt} (see Corollary %
\ref{cgt}).
\end{itemize}

\begin{enumerate}
\item[(2)] \noindent $\pi =1$ ($\alpha $ \textbf{injective)}.
\end{enumerate}

\begin{itemize}
\item When $X=S^{n}$ ($n\geq 5$), $R=\mathbb{Z}$ and $G$ a superperfect
group, Theorem \ref{th3} recovers the existence of homology spheres, which
is obtained by Kervaire \cite{kar} (see Corollary \ref{ker}).

\item When $X=S^{n}$ ($n\geq 5$), $R=\mathbb{Z}$ and $H_{1}(G;\mathbb{Z})=\mathbb{Z}$, $%
H_{2}(G;\mathbb{Z})=0,$ Theorem \ref{th3} recovers the existence higher-dimensional
knots, which is obtained by Kervaire \cite{ker} (see Corollary \ref{knot}).

\item When $X=S^{n}$ ($n\geq 6$) and $R=C_{\mathbb{R}}^{\ast }(G),$ the
theorem yields the results obtained by Farber-Weinberger \cite{FW} and
Higson-Roe-Schick \cite{hig} on the zero-in-the-spectrum conjecture (see
Corollary \ref{zero}).
\end{itemize}

For Bousfield's integral localization, Rodr\'{\i}guez and Scevenels \cite{rs}
show that there is a topological construction that, while leaving the
integral homology of a space unchanged, kills the intersection of the
transfinite lower central series of its fundamental group. Moreover, this is
the maximal subgroup that can be factored out of the fundamental group
without changing the integral homology of a space. As another application of
Theorem \ref{th3} with $\alpha $ surjective and $R=\mathbb{Z}$, we obtain a
manifold version of Rodr\'{\i}guez and Scevenels' result.

\begin{corollary}
\label{integral}Let $n\geq 5$ and $X$ be a closed $n$-dimensional spin
manifold with fundamental group $\pi $ and $N$ a normal subgroup of $\pi .$
The following are equivalent.

\begin{enumerate}
\item[(i)] There exists a closed manifold $Y$ obtained from $X$ by adding $2$%
-handles and $3$-handles with $\pi _{1}(Y)=\pi /N,$ such that for any $q\geq
0$ there is an isomorphism%
\begin{equation*}
H_{q}(Y;\mathbb{Z})\cong H_{q}(X;\mathbb{Z}).
\end{equation*}

\item[(ii)] The group $N$ is normally generated by finitely many elements
and is a relatively perfect subgroup of $\pi ,$ \textsl{i.e.} $[\pi ,N]=N.$
\end{enumerate}
\end{corollary}

The article is organized as follows. In Section 2, we introduce some basic
facts about finitely $G$-dense rings, surgery theory, Poincar\'{e} duality
with coefficients and one-sided homology cobordism. The main theorem is
proved in Section 3 and some applications are given in Section 4.

\section{Preliminary results and basic facts}

\subsection{Finitely $G$-dense rings}

Recall the concept of finitely $G$-dense rings in Definition \ref{fgdense}
(see also Ye \cite{er}, Definition 1). Compared with the definition of $G$%
-dense rings, we require all the modules in Definition \ref{fgdense} to be
finitely generated. It is clear that a $G$-dense ring is finitely $G$-dense.
The following lemma from Ye \cite{er} gives some typical examples of finitely $%
G$-dense rings.

\begin{lemma}[ \protect Ye \cite{er}, Lemma 2]
Finitely $G$-dense rings include the real reduced group $C^{\ast }$-algebra $%
C_{\mathbb{R}}^{\ast }(G),$ the real group von Neumann algebra $\mathcal{N}_{%
\mathbb{R}}G,$ the real Banach algebra $l_{\mathbb{R}}^{1}(G),$ the rings $k=%
\mathbb{Z}/p$ (prime $p$) and $k\subseteq \mathbb{Q}$ a subring of the
rationals (with trivial $G$-actions), the group ring $k[G]$.
\end{lemma}

Similar to Example 2.6 in Ye \cite{Ye}, one can show that the ring of Gaussian
integers $\mathbb{Z[}i\mathbb{]}$ is not finitely $G$-dense for the trivial
group $G$.

\subsection{Basic facts on surgery}

The proof of Theorem \ref{th3} is based on some facts in surgery theory. The
following definition and lemmas can be found in Ranicki \cite{rnk}.

\begin{definition}
\label{def}Given an $n$-manifold $M$ and an embedding $S^{i}\times
D^{n-i}\subset M$ $(-1\leq i\leq n)$ define the $n$-manifold $M^{\prime
}=(M-S^{i}\times D^{n-i})\cup D^{i+1}\times S^{n-i-1}$ obtained from $M$ by
surgery by an $i$-surgery. The \emph{trace} of the surgery on $S^{i}\times
D^{n-i}\subset M$ is the elementary $(n+1)$-dimensional cobordism $%
(W;M,M^{\prime })$ obtained from $M\times \lbrack 0,1]$ by attaching an $%
(i+1)$-handle $W=M\times \lbrack 0,1]\cup _{S^{i}\times D^{n-i}\times {1}%
}D^{i+1}\times D^{n-i}.$
\end{definition}

The following lemma gives homotopy relations between the surgery trace and
the manifolds.

\begin{lemma}
\label{hom}Let $M,M^{\prime }$ and $W$ be the manifolds defined in
Definition \ref{def}. There are homotopy equivalences
\begin{equation*}
M\cup e^{i+1}\simeq W\simeq M^{\prime }\cup e^{n-i},
\end{equation*}%
where attaching maps are induced by embedding of handles.
\end{lemma}

For a manifold $M,$ denote by $w(M)$ the first Stiefel-Whitney class of the
tangent bundle over $M$. The following lemma is Proposition 4.24 in Ranicki \cite%
{rnk}.

\begin{lemma}
\label{lem} Let $(W;M,M^{\prime })$ be the trace of an $n$-surgery on an $m$%
-dimensional manifold $M$ killing $x\in H_{n}(M)$.

1) If $1\leq n\leq m-2$, then $W$ and $M^{\prime }$ has the same orientation
type as $M$ (which means that $M^{\prime }$ is orientable iff $M$ is
orientable).

2) If $n=m-1$ and $M$ is orientable, so are $M^{\prime }$ and $W.$

3) If $n=m-1$ and $M$ is nonorientable, then $M^{\prime }$ is orientable if
and only if $x=w(M)\in H_{m-1}(M;\mathbb{Z}_{2})=H^{1}(M;\mathbb{Z}_{2}).$

4) If $n=0$ and $M$ is nonorientable, then so are $W$ and $M^{\prime }$.
\end{lemma}

\subsection{Poincar\'{e} duality with coefficients}

In this subsection, we collect some facts about the Poincar\'{e} duality
with coefficients. For more details, see Chapter 2 of Wall's book \cite{wall}%
. Let $X$ be a finite CW complex with a universal covering space $\tilde{X}.$
Denote by $C_{\ast }(\tilde{X})$ the cellular chain complex of $X$ and by $%
C^{\ast }(\tilde{X})$ the chain complex $\mathrm{Hom}_{\mathbb{Z}\pi
_{1}(X)}(C_{\ast }(\tilde{X}),\mathbb{Z}\pi _{1}(X)).$ We call a finite CW
complex $X$ a simple Poincar\'{e} complex if for some positive integer $n$
and a representative cycle $\xi \in C_{n}(\tilde{X})\otimes_{\mathbb{Z\pi
}_{1}(X)}\mathbb{Z}$, the cap product induces a chain homotopy equivalence%
\begin{equation*}
\xi \cap :C^{\ast }(\tilde{X})\rightarrow C_{n-\ast }(\tilde{X})
\end{equation*}%
and the Whitehead torsion is vanishing. Similarly, we can define Poincar\'{e}
pair $(Y,X)$ by a simple homotopy equivalence%
\begin{equation*}
\xi \cap :C^{\ast }(\tilde{Y})\rightarrow C_{n-\ast }(\tilde{Y},\tilde{X}).
\end{equation*}%
When $X$ is an $n$-dimensional closed manifold, $X$ is a simple Poincar\'{e}
complex of formal dimension $n.$ When $X$ is a compact manifold with
boundary $\partial X$, the pair $(X,\partial X)$ is a simple Poincar\'{e}
pair (cf. Theorem 2.1 on page 23 in Wall \cite{wall}).

\begin{lemma}
\label{poin}Let $M$ be an $n$-dimensional orientable compact manifold with
boundary $\partial M=X\dot{\cup}Y$ for closed manifolds $X$ and $Y.$ Then
for any integer $q\geq 0$ and any $\mathbb{Z}\pi _{1}(M)$-module $R,$ there
is an isomorphism%
\begin{equation*}
H^{q}(M,X;R)\rightarrow H_{n-q}(M,Y;R).
\end{equation*}
\end{lemma}

\begin{proof}
Since $X$ is a Poincar\'{e} complex and $(M,X\dot{\cup}Y)$ is a Poincar\'{e}
pair, the lemma can be proved by considering the long exact homology and
cohomology sequences for the cofiber sequence of pairs%
\begin{equation*}
(X,\emptyset )\rightarrow (M,Y)\rightarrow (M,X\dot{\cup}Y)
\end{equation*}%
using Poincar\'{e} duality for $X$ and the pair $(M,X\dot{\cup}Y).$ When $R$
is commutative, this is Theorem 3.43 of Hatcher's textbook \cite{hat}. The
proof of general case is similar.
\end{proof}

\subsection{One-sided $R$-homology cobordism}

Recall from Guilbault and Tinsley \cite{gt} that a \textit{one-sided h-cobordism} $(W;X,Y)$ is a
compact cobordism between closed manifolds such that $Y\hookrightarrow W$ is
a homotopy equivalence. Motivated by this, we can define one-sided homology
cobordism.

\begin{definition}
Let $(W;X,Y)$ be a compact cobordism between closed manifolds and $R$ a $%
\mathbb{Z}\pi _{1}(W)$-module. We call $(W;X,Y)$ a \emph{one-sided }$R$\emph{%
-homology cobordism} if the inclusion $Y\hookrightarrow W$ induces $\pi
_{1}(Y)\overset{\cong }{\rightarrow }\pi _{1}(W)$ and, for any integer $%
q\geq 0$ an isomorphism $H_{q}(Y;R)\cong H_{q}(W;R).$
\end{definition}

The following are some easy facts.

\begin{enumerate}
\item[(1)] When $R=\mathbb{Z[}\pi _{1}(W)],$ one-sided $R$-homology
cobordism is the same as one-sided $h$-cobordism.
\end{enumerate}

\begin{proof}
By the Whitehead theorem, the homotopy equivalence follows from the homology
equivalence with coefficients $\mathbb{Z[}\pi _{1}(W)]$ and the isomorphism
of fundamental groups.
\end{proof}

\begin{enumerate}
\item[(2)] Let $(W;X,Y)$ be a one-sided $R$-homology cobordism. For any
integer $q\geq 0,$ the inclusion map induces an isomorphism $H^{q}(W;R)\cong
H^{q}(X;R).$
\end{enumerate}

\begin{proof}
This follows directly from Poincar\'{e} duality with coefficients as in the
previous subsection.
\end{proof}

\begin{enumerate}
\item[(3)] For a one-sided $h$-cobordism $(W;X,Y),$ the inclusion map $%
X\hookrightarrow W$ is a Quillen's plus construction.
\end{enumerate}

\begin{proof}
Since $Y\hookrightarrow W$ is a homotopy equivalence, we have that for any
integer $q\geq 0,$ the relative cohomology group $H^{q}(W,Y;\mathbb{Z[}\pi
_{1}(W)])=0.$ By Poincar\'{e} duality, the inclusion map $X\hookrightarrow W$
induces homology isomorphism with coefficients $\mathbb{Z[}\pi _{1}(W)].$
According to 4.3 xi in Berrick \cite{berr}, the inclusion map is then a Quillen's
plus construction.
\end{proof}

\begin{enumerate}
\item[(4)] Let $R$ be a principal ideal domain and $(W;X,Y)$ a one-sided $R$%
-homology cobordism. Then for any integer $q\geq 0,$ there is an isomorphism
$H_{q}(X;R)\cong H_{q}(Y;R).$
\end{enumerate}

\begin{proof}
When the inclusion map $Y\hookrightarrow W$ induces an $R$-homology
equivalence, it also induces an $R$-cohomology equivalence by the universal
coefficients theorem. By Poincar\'{e} duality, $X$ has the same homology as $%
W,$ also as $Y.$
\end{proof}

\section{Proof of Theorem \protect\ref{th3}}

In this section, we will prove Theorem \ref{th3}. First, we need some facts
about finitely presented groups.

Recall that a normal subgroup $N$ of a group $\pi $ is called \emph{normally
finitely generated} if there exists a finite set $S\subset N$ such that $N$
is generated by elements of the form $gsg^{-1}$ for $s\in S$ and $g\in \pi .$
The following lemma gives an elementary characterization of when a normal
subgroup is normally finitely generated. Since it is helpful for our later
argument, we present a short proof here.

\begin{lemma}
\label{fp}Let $\pi $ be a finitely presented group and $N$ a normal
subgroup. Then $N$ is normally finitely generated if and only if $\pi /N$ is
finitely presented.
\end{lemma}

\begin{proof}
The necessity of the condition is obvious. Conversely, choose a presentation
of $\pi /N$ with finitely many generators and finitely many relations. Let $%
F_{n}$ be the free group with $n$ generators and $f:F_{n}\rightarrow \pi $ a
surjection, with normally finitely generated kernel $K.$ We can also assume
that that the generators of $F_{n}$ are mapped surjectively to the
generators of $\pi /N$ by the composition of $f$ with the quotient map $\pi
\rightarrow \pi /N.$ Here we use the fact that the condition that a group is
finitely presented does not depend on the choice of a generator system (cf.
Prop. 1.3 in Ohshika \cite{dg}). Since $\pi /N$ is finitely presented, $f^{-1}(N)$
is normally finitely generated. Then $N=f(f^{-1}(N))$ is normally finitely
generated.
\end{proof}

In order to prove Theorem \ref{th3}, we use the following lemma, which is a
more general version of Hopf's exact sequence.

\begin{lemma}[Lemma 2.2 in \protect Higson-Roe-Schick \cite{hig}]
\label{lem1}Let $G$ be a group and $V$ be a left $\mathbb{Z}[G]$-module. For
any CW complex $X$ with fundamental group $G$ and universal covering space $%
\tilde{X}$, there is an exact sequence%
\begin{equation*}
H_{2}(\tilde{X})\bigotimes_{\mathbb{Z}[G]}V\rightarrow
H_{2}(X;V)\rightarrow H_{2}(G;V)\rightarrow 0.
\end{equation*}
\end{lemma}

\noindent \textbf{Proof of Theorem \ref{th3}} We construct a manifold $Y_{1}$
whose fundamental group $\pi _{1}(Y)=G$ as follows. Fix a finite
presentation
\begin{equation*}
\langle x_{1},x_{2},\ldots ,x_{k}|y_{1},y_{2},\ldots ,y_{l}\rangle
\end{equation*}%
of $\alpha (\pi ).$ Extend the presentation of $\alpha (\pi )$ by generators
and relations to yield a presentation
\begin{equation*}
\langle x_{1},x_{2},\ldots ,x_{k},g_{1},g_{2},\ldots
,g_{u}|y_{1},y_{2},\ldots ,y_{l},r_{1},r_{2},\ldots ,r_{v}\rangle
\end{equation*}%
of $G$ by adding some generators and relations. For adding the generators $%
g_{1},g_{2},\ldots ,g_{u}$, let $S_{i}^{0}$ be a copy of the $0$-sphere $%
S^{0}$ and
\begin{equation*}
f_{1}:\amalg _{i=1}^{u}S_{i}^{0}\times D^{n}\rightarrow X
\end{equation*}%
be an embedding with disjoint image. Add $1$-handles along $f_{1}$ to $X.$
The resulting manifold is $X_{1}$ and denote by $W_{1}$ the surgery trace.
We see that the manifold $X_{1}$ is actually the connected sum $X\sharp
_{i=1}^{u}S^{1}\times S^{n-1}$ and can have the same orientation type as $X.$
Denote by $e_{i}^{j}$ a copy of $j$-cells indexed by $i$. By Lemma \ref{hom}%
, there are homotopy equivalences%
\begin{equation*}
X\cup _{i=1}^{u}e_{i}^{1}\simeq W_{1}\simeq X_{1}\cup _{i=1}^{u}e_{i}^{n}.
\end{equation*}%
According to the construction, the fundamental group of $X_{1}$ is
\begin{equation*}
\pi _{1}(X_{1})=\pi \ast F(g_{1},g_{2},\ldots ,g_{u}),
\end{equation*}%
the free product of $\pi =\pi _{1}(X)$ and the free group of $u$ generators.
By Lemma \ref{fp}, the kernel $\ker \alpha $ is normally generated by
finitely many elements $z_{1},z_{2},\ldots ,z_{p}$. Denote by $S$ the finite
set $\{z_{1},z_{2},\ldots ,z_{p},r_{1},r_{2},\ldots ,r_{v}\}.$ Choose as
usual a contractible open set $U$ in $X_{1}$ as "base point". According to
Whitney's theorem, any element in $\pi _{1}(X_{1})$ is represented by an
embedded $1$-sphere. Since the manifold $X_{1}$ is orientable, the normal
bundle of any embedded $1$-sphere is trivial. For the elements in $S$, let $%
S_{\lambda }^{1}$ be a copy of the $1$-sphere $S^{1}$ and let%
\begin{equation*}
f_{2}:\amalg _{\lambda \in S}S_{\lambda }^{1}\times D^{n-1}\rightarrow X_{1}
\end{equation*}%
be disjoint embeddings representing the corresponding elements in $\pi
_{1}(X_{1}).$ Do surgery along $f_{2}$ by attaching $2$-handles. The
resulting manifold is $X_{2},$ and denote by $W_{2}$ the surgery trace. By
Lemma \ref{hom} once again, there are homotopy equivalences%
\begin{equation*}
X_{1}\cup _{\lambda \in S}e_{\lambda }^{2}\simeq W_{2}\simeq X_{2}\cup
_{\lambda \in S}e_{\lambda }^{n-1}.
\end{equation*}%
Since $n\geq 4,$ the fundamental group of $X_{2}$ is $G.$ Let $W^{\prime }$
be the manifold obtained by gluing the two traces $W_{1}$ and $W_{2}$
together along $X_{1}.$ There are homotopy equivalences%
\begin{eqnarray*}
W^{\prime } &\simeq &X\cup _{i=1}^{u}e_{i}^{1}\cup _{\lambda \in
S}e_{\lambda }^{2}\simeq X_{1}\cup _{i=1}^{u}e_{i}^{n}\cup _{\lambda \in
S}e_{\lambda }^{2} \\
&\simeq &X_{2}\cup _{i=1}^{u}e_{i}^{n}\cup _{\lambda \in S}e_{\lambda
}^{n-1}.
\end{eqnarray*}%
This implies the fundamental group of $W^{\prime }$ is also $G,$ since $n>3.$

We consider the homology groups of the pair $(W^{\prime },X).$ Let $\tilde{X}
$ and $\tilde{W}^{\prime }$ be the universal covering spaces of $X$ and $%
W^{\prime }.$ By Lemma \ref{lem1}, there is a commutative diagram
\begin{equation*}
\begin{array}{cccc}
H_{2}(\tilde{X})\bigotimes\nolimits_{\mathbb{Z}G}R & \rightarrow & H_{2}(%
\tilde{W}^{\prime })\bigotimes\nolimits_{\mathbb{Z}G}R &  \\
\downarrow &  & \downarrow j_{4} &  \\
H_{2}(X;R) & \overset{j_{2}}{\longrightarrow } & H_{2}(W^{\prime };R) &
\overset{j_{1}}{\rightarrow }H_{2}(W^{\prime },X;R)\rightarrow
H_{1}(X^{\prime };R)\rightarrow H_{1}(W^{\prime };R) \\
\downarrow j_{3} &  & \downarrow j_{5} &  \\
H_{2}(\pi ;R) & \overset{\alpha _{\ast }}{\longrightarrow } & H_{2}(G;R) &
\end{array}%
\end{equation*}%
where the middle horizontal chain is the long exact sequence of homology
groups for the pair $(W^{\prime },X)$ and the two vertical lines are the
Hopf exact sequences as in Lemma \ref{lem1}. Notice that
\begin{equation*}
H_{1}(X;R)\cong H_{1}(\pi ;R)\rightarrow H_{1}(W^{\prime };R)\cong H_{1}(G;R)
\end{equation*}%
is injective by assumption. This implies $j_{1}:H_{2}(W^{\prime
};R)\rightarrow H_{2}(W^{\prime },X;R)$ is surjective in the above diagram.
Note that the map $\alpha _{\ast }:H_{2}(\pi ;R)\rightarrow H_{2}(G;R)$ is
surjective by assumption. By a diagram chase (for more details, see the
proof of Theorem 1.1 in Ye \cite{Ye}), there is a surjection
\begin{equation*}
j_{1}\circ j_{4}:H_{2}(\tilde{W}^{\prime })\bigotimes\nolimits_{\mathbb{Z}%
G}R\rightarrow H_{2}(W^{\prime },X;R).
\end{equation*}%
As $\tilde{W}^{\prime }$ is simply connected, the homology group $H_{2}(%
\tilde{W}^{\prime })$ is isomorphic to $\pi _{2}(W^{\prime })\cong \pi
_{2}(X_{2}).$ Notice that the homology group $H_{2}(W^{\prime },X;R)$ can be
taken to be a finitely generated free $R$-module as in the proof of Theorem
1 in Ye \cite{Ye}. Since the ring $R$ is a finitely $G$-dense ring in the sense
of Definition \ref{fgdense}, we can find a finite set $S^{\prime }$ of
elements in $\pi _{2}(X_{2})$ such that the image $j_{1}\circ
j_{4}(S^{\prime })$ forms an $R$-basis for $H_{2}(W^{\prime },X;R).$ Then
there are maps $b_{\lambda }:S_{\lambda }^{2}\rightarrow X_{2}$ with $%
\lambda \in S^{\prime }$ such that for all $q\geq 2,$ the composition of maps%
\begin{equation*}
H_{q}(\vee _{\lambda \in S^{\prime }}S_{\lambda }^{2};R)\rightarrow
H_{q}(W^{\prime };R)\rightarrow H_{q}(W^{\prime },X;R)
\end{equation*}%
is an isomorphism$.$

We construct the manifolds $Y$ and $W$ as follows. Notice that an embedded $%
2 $-sphere in a $k$-dimensional $(k\geq 5)$ orientable manifold $M$ has
trivial normal bundle if and only if it represents $0$ in $\pi _{1}(SO(k))=%
\mathbb{Z/}2$ through the classifying map $M\rightarrow BSO(k)$ (cf. page 45
of Milnor \cite{mi}). When $2$ is invertible in the ring $R,$ we can always choose
the $2$-spheres in $S^{\prime }$ to have trivial normal bundles. When $2$ is
not invertible in $R,$ the manifold $X$ is a spin manifold by assumption.
This implies any embedded $2$-sphere has a trivial normal bundle. Since $%
n\geq 5,$ we can choose a map
\begin{equation*}
f_{3}:\amalg _{\lambda \in S^{\prime }}S_{\lambda }^{2}\times
D^{n-2}\rightarrow X_{2}
\end{equation*}%
as disjoint embedding, whose components represent the elements $b_{\lambda }$%
. Do surgery along $f_{3}$ by attaching $3$-handles. Let $Y$ denote the
resulting manifold and $W_{3}$ denote the surgery trace. Suppose that $W$ is
the manifold obtained by gluing $W^{\prime }$ and $W_{2}$ along $X_{2},$
which is a cobordism between $X$ and $Y.$ By Lemma \ref{hom}, there are
homotopy equivalences%
\begin{equation*}
X_{2}\cup _{\lambda \in S^{\prime }}e_{\lambda }^{3}\simeq W_{3}\simeq Y\cup
_{\lambda \in S^{\prime }}e_{\lambda }^{n-2}
\end{equation*}%
and
\begin{equation*}
W\simeq W^{\prime }\cup _{\lambda \in S^{\prime }}e_{\lambda }^{3}\simeq
Y\cup _{i=1}^{u}e_{i}^{n}\cup _{\lambda \in S}e_{\lambda }^{n-1}\cup
_{\lambda \in S^{\prime }}e_{\lambda }^{n-2}.
\end{equation*}%
By the van Kampen theorem, the fundamental group of $Y$ is still $G,$ since $%
n>4.$ Denoting by $H_{\ast }(-)$ the homology groups $H_{\ast }(-;R),$ we
have the following commutative diagram:%
\begin{equation*}
\begin{array}{ccccccc}
\cdots \rightarrow H_{3}(\vee D^{3},\vee S^{2}) & \rightarrow & H_{2}(\vee
S^{2},\mathrm{pt}) & \rightarrow & H_{2}(\vee D^{3},\mathrm{pt}) &
\rightarrow & H_{2}(\vee D^{3},\vee S^{2}) \\
\downarrow &  & \downarrow &  & \downarrow &  & \downarrow \\
\cdots \rightarrow H_{3}(W,W^{\prime }) & \rightarrow & H_{2}(W^{\prime },X)
& \rightarrow & H_{2}(W,X) & \rightarrow & H_{2}(W,W^{\prime }).%
\end{array}%
\end{equation*}%
By a five lemma argument and the assumption that $\alpha _{\ast }:H_{1}(\pi
;R)\rightarrow H_{1}(G;R)$ is injective, for any integer $q\geq 2,$ the
relative homology group $H_{q}(W,X;R)=0.$ This shows for any integer $q\geq
2,$ the homology groups $H_{q}(X;R)\cong H_{q}(W;R)$ and proves the
isomorphism in (1).

\begin{remark}
\label{spin}From the proof, we can see that for some special group
homomorphism $\alpha $ and coefficients $R,$ the orientability or spin-ness
of $X$ in Theorem \ref{th3} can be dropped. For example, when $\alpha $ is
surjective and $\ker (\alpha )<[\pi ,\pi ]$, we do not need $X$ to
orientable. When $\ker (\alpha )$ is perfect (or weakly $L$-perfect for some
normal group), the spin-ness of $X$ can be dropped (cf. the proof of Theorem
4.1 and Theorem 5.2 in Guilbault and Tinsley \cite{gt}).
\end{remark}

\section{Applications}

In this section, we give several applications of Theorem \ref{th3}.

\subsection{Surgery plus construction for manifolds}

In this subsection, we get a manifold version of Quillen's plus construction
by doing surgery. The following proposition is a special case of Theorem \ref%
{th3} when $\pi $ has a normally finitely generated normal perfect subgroup $%
P$, $G=\pi /P$ and $R=\mathbb{Z[}G].$

\begin{corollary}
\label{plus}Let $n\geq 5$ be an integer and $M$ a closed $n$-dimensional
spin manifold. Suppose that $P$ is a normal perfect subgroup in $\pi _{1}(M)$
normally generated by some finite elements. Then there exists a one-sided
h-cobordism $(W;M,Y)$ such that $\pi _{1}(W)=\pi _{1}(M)/P$. More precisely,
there exists a closed spin manifold $Y$ with the following properties:

\begin{enumerate}
\item[(i)] $Y$ is obtained from $M$ by attaching 2-handles and 3-handles,
such that

\item[(ii)] $\pi _{1}(Y)=\pi _{1}(M)/P$ and the inclusion map $%
g:M\rightarrow W,$ the cobordism between $X$ and $Y$, is Quillen's plus
construction inducing epimorphism $\pi _{1}(M)\rightarrow \pi _{1}(M)/P$ of
fundamental groups; and

\item[(iii)] $Y$ has the homotopy type of Quillen's plus construction $%
M^{+}. $
\end{enumerate}
\end{corollary}

\begin{proof}
We apply Theorem \ref{th3}. Let $X=M$ and $\alpha :\pi =\pi
_{1}(M)\rightarrow \pi /P$ be the quotient map. By Hilton and Stammbach \cite{h}, there is an
exact sequence%
\begin{eqnarray*}
H_{2}(\pi ;\mathbb{Z}[\pi /P]) &\rightarrow &H_{2}(\pi /P;\mathbb{Z}[\pi
/P])\rightarrow \mathbb{Z}[\pi /P]\bigotimes\nolimits_{\mathbb{Z}[\pi
/P]}P_{\mathrm{ab}} \\
&\rightarrow &H_{1}(\pi ;\mathbb{Z}[\pi /P])\rightarrow H_{1}(\pi /P;\mathbb{%
Z}[\pi /P])\rightarrow 0.
\end{eqnarray*}%
When $\mathbb{Z}[\pi /P]\bigotimes_{\mathbb{Z}[\pi /P]}P_{\mathrm{ab}}\cong
\mathbb{Z}\bigotimes_{\mathbb{Z}}P_{\mathrm{ab}}=0$, we can see
\begin{equation*}
H_{2}(\pi ;\mathbb{Z}[\pi /P])\rightarrow H_{2}(\pi /P;\mathbb{Z}[\pi /P])
\end{equation*}%
is surjective and $H_{1}(\pi ;\mathbb{Z}[\pi /P])\rightarrow H_{1}(\pi /P;%
\mathbb{Z}[\pi /P])$ is an isomorphism. Therefore, the conditions of group
homomorphism $\alpha $ are satisfied. By Theorem \ref{th3}, there exists a
closed spin manifold $Y$ and cobordism $(W;M,Y)$ such that for any integer $%
q\geq 0,$ there is an isomorphism
\begin{equation*}
H_{q}(M;\mathbb{Z[}G])\cong H_{q}(W;\mathbb{Z[}G]).
\end{equation*}%
This implies the inclusion map $g:X\rightarrow W$ is the Quillen plus
construction (cf. 4.3 xi in Berrick \cite{berr}). Therefore, for all integers $q,$
the relative cohomology groups $H^{q}(W,X;\mathbb{Z[}G])=0.$ According to
the Poincar\'{e} duality in Lemma \ref{poin}, for each integer $q\geq 0,$
the relative homology group $H_{q}(W,Y;\mathbb{Z[}G])=0$ and there is an
isomorphism%
\begin{equation*}
H_{q}(Y;\mathbb{Z[}G])\cong H_{q}(W;\mathbb{Z[}G])
\end{equation*}%
as well. This means the universal covering spaces of $Y$ and $W$ are
homology equivalent and therefore also homotopy equivalent. Since $Y$ and $W$
have the same fundamental group, this implies the inclusion map $%
Y\rightarrow W$ is a homotopy equivalence. This finishes the proof.
\end{proof}

Corollary \ref{plus} was first obtained by Hausmann \cite{ha} (see also
Hausmann \cite{ha2} and Guilbault-Tinsley \cite{gt2, gt}).

Recall the definition of (mod $L$)-one-sided $h$-cobordism by Guilbault and
Tinsley \cite{gt}. We see that a (mod $L$)-one-sided $h$-cobordism is a
one-sided $\mathbb{Z}[\pi _{1}(Y)/L]$-homology cobordism. The following
result proved by Guilbault and Tinsley \cite{gt} is a corollary of Theorem %
\ref{th3} when $R=\mathbb{Z}[\pi _{1}(Y)/L].$ Note that $\mathbb{Z}[\pi
_{1}(Y)/L]$ is a finitely $\pi _{1}(Y)$-dense ring.

\begin{corollary}[\protect Guilbault-Tinsley \cite{gt}, Theorem 5.2 ]
\label{cgt}Let $B$ be a closed $n$-manifold $(n\geq 5)$ and $\alpha :\pi
_{1}(B)\rightarrow G$ a surjective homomorphism onto a finitely presented
group such that $\ker (\alpha )$ is strongly $L^{\prime }$-perfect, \textsl{%
i.e.} $\ker (\alpha )=[\ker (\alpha ),L^{\prime }]$ for some group $%
L^{\prime },$ where $\ker (\alpha )\trianglelefteq L^{\prime
}\trianglelefteq \pi _{1}(B)$ and all loops representing elements in $%
L^{\prime }$ are orientation-preserving. Then for $L=L^{\prime }/\ker
(\alpha ),$ there exists a (mod $L$)-one-sided h-cobordism $(W;B,A)$ such
that $\pi _{1}(W)=G$ and $\ker (\pi _{1}(B)\rightarrow \pi _{1}(W))=\ker
(\alpha ).$
\end{corollary}

\begin{proof}
The proof is similar to that of Corollary \ref{plus}. When $\ker (\alpha
^{\prime })$ is $L^{\prime }$-perfect, we have that $\mathbb{Z}%
[G/L]\bigotimes_{\mathbb{Z[}G\mathbb{]}}\ker (\alpha )_{\mathrm{ab}}=0.$
According to the 5-term exact sequence for group homology (cf. $(8.2)$ in Hilton and Stammbach
\cite{h}, page 202)%
\begin{gather*}
H_{2}(\pi _{1}(B);\mathbb{Z}[G/L])\overset{H_{2}(\alpha )}{\rightarrow }%
H_{2}(G;\mathbb{Z}[G/L])\rightarrow \mathbb{Z}[G/L]\bigotimes\nolimits_{\mathbb{Z}[G]}\ker (\alpha )_{\mathrm{ab}} \\
\rightarrow H_{1}(\pi _{1}(B);k[G/L])\overset{H_{1}(\alpha )}{\rightarrow }%
H_{1}(G;k[G/L])\rightarrow 0,
\end{gather*}%
we can see that $H_{2}(\alpha )$ is surjective and $H_{1}(\alpha )$ is
isomorphic. By Theorem \ref{th3} with $R=\mathbb{Z}[G/L]$ and the remark
followed, there exists a cobordism $(W;A,B)$ such that $\pi _{1}(B)=\pi
_{1}(W)=G$ and the inclusion $B\hookrightarrow W$ induces homology
equivalence with coefficients $R$. Considering the covering spaces of $B$
and $W$ with deck transformation group $G/L,$ we can see that the inclusion $%
B\hookrightarrow W$ also induces a cohomology equivalence with coefficients $%
R$. By Poincar\'{e} duality in Lemma \ref{poin}, the inclusion $%
A\hookrightarrow W$ induces homology equivalence with coefficients $R$. This
finishes the proof.
\end{proof}

\subsection{Surgery preserving integral homology groups}

In this subsection, we study the case when the integral homology groups of a
manifold are preserved by doing surgery. Corollary \ref{integral} is a
special case of Theorem \ref{th3} when $R=\mathbb{Z}$.

\begin{proof}[Proof of Corollary \protect\ref{integral}]
We show that (ii) implies (i) first. By Hilton and Stammbach \cite{h}, there is an exact sequence%
\begin{equation}
H_{2}(\pi ;\mathbb{Z})\rightarrow H_{2}(\pi /N;\mathbb{Z})\rightarrow N/[\pi
,N]\rightarrow H_{1}(\pi ;\mathbb{Z})\rightarrow H_{1}(\pi /N;\mathbb{Z}%
)\rightarrow 0.  \tag{3}
\end{equation}%
When $N=[\pi ,N]$, we have that the map $H_{2}(\pi ;\mathbb{Z})\rightarrow
H_{2}(\pi /N)$ is surjective and $H_{1}(\pi ;\mathbb{Z})\rightarrow
H_{1}(\pi /N)$ is an isomorphism. According to Theorem \ref{th3} with $R=%
\mathbb{Z}$, there exists a closed spin manifold $Y$ obtained from $X$ by
adding $2$-handles and $3$-handles with $\pi _{1}(Y)=\pi _{1}(X)/N.$ Let $W$
be the cobordism between $X$ and $Y.$ Furthermore, we have for any $q\geq 0$
there is an isomorphism $H_{q}(X;\mathbb{Z})\cong H_{q}(W;\mathbb{Z}).$
According to the universal coefficients theorem, for all integers $q\geq 0$
the relative cohomology groups $H^{q}(W,X;\mathbb{Z})=0.$ Therefore by
Theorem \ref{th3} (ii), for any integer $q\geq 0$ there is an isomorphism%
\begin{equation*}
H_{q}(Y;\mathbb{Z})\cong H_{q}(X;\mathbb{Z}).
\end{equation*}

Conversely, suppose $(W;X,Y)$ is a cobordism with the boundary $X$ and $Y.$
Since $Y$ is obtained from $X$ by doing surgery below the middle dimension,
we have isomorphisms $H_{1}(Y;\mathbb{Z})\cong H_{1}(W;\mathbb{Z})$ and $%
H_{2}(Y;\mathbb{Z})\cong H_{2}(W;\mathbb{Z})\cong H_{2}(X;\mathbb{Z})$.
Therefore the inclusion map $X\rightarrow W$ induces an isomorphism
\begin{equation*}
H_{1}(\pi ;\mathbb{Z})=H_{1}(X;\mathbb{Z})\cong H_{1}(W;\mathbb{Z}%
)=H_{1}(\pi _{1}(W);\mathbb{Z}).
\end{equation*}%
According to the Hopf exact sequence (cf. Lemma \ref{lem1}), there is a
commutative diagram%
\begin{equation*}
\begin{array}{ccc}
H_{2}(X;\mathbb{Z}) & \twoheadrightarrow & H_{2}(\pi ;\mathbb{Z}) \\
\downarrow &  & \downarrow \\
H_{2}(W;\mathbb{Z}) & \twoheadrightarrow & H_{2}(\pi _{1}(Y);\mathbb{Z})%
\end{array}%
\end{equation*}%
where the left vertical map is an isomorphism. This shows that the right
vertical map is an epimorphism. According to the same exact sequence (3)
above, we have $N=[\pi ,N],$ which means $N$ is relative perfect. Since $\pi
_{1}(X)/N=\pi _{1}(Y)$ is finitely presented, $N$ is normally finitely
generated by Lemma \ref{fp}.
\end{proof}

Corollary \ref{integral} is a manifold version of a result obtained by Rodr%
\'{\i}guez and Scevenels \cite{rs} for CW complexes.

\subsection{The fundamental groups of homology manifolds}

In this subsection, we study the fundamental groups of manifolds with the
same homology type as a $2$-connected manifold.

\begin{theorem}
\label{fund}Let $G$ be a finitely presented group, $R$ a subring of the
rationals or the constant ring $\mathbb{Z}/p$ (prime $p$) and $n\geq 5$ an
integer. Suppose $M$ is a $2$-connected manifold of dimension $n$. Then the
following are equivalent:

\begin{enumerate}
\item[(i)] There exists an $n$-dimensional manifold $Y$ obtained from $M$ by
adding $1$-handles, $2$-handles and $3$-handles with $\pi _{1}(Y)=G$ and for
any integer $q\geq 0,$ we have%
\begin{equation*}
H_{q}(Y;R)\cong H_{q}(M;R);
\end{equation*}

\item[(ii)] The group $G$ is $R$-superperfect, i.e. $H_{1}(G;R)=0$ and $%
H_{2}(G;R)=0.$
\end{enumerate}
\end{theorem}

\begin{proof}
We show that (i) implies (ii) first. By assumptions, we have $%
H_{2}(M;R)=H_{1}(M;R)=0.$ If some manifold $Y$ has the homology groups with
coefficients $R$ as $M,$ then $H_{1}(G;R)\cong H_{1}(M;R)=0.$ According to
the Hopf exact sequence in Lemma \ref{lem1}, we get $H_{2}(G;R)=0.$

Conversely, suppose that $G$ is a finitely presented group with $%
H_{2}(G;R)=H_{1}(G;R)=0.$ Let $X=M,$ $\pi =1,$ the trivial group and $f:\pi
\rightarrow G$ the obvious group homomorphism. Note that the $2$-connected
manifold $X$ is always a spin manifold and $R$ is a principal ideal domain.
According to Theorem\textbf{\ }\ref{th3}\textbf{, }we get a manifold $Y$
with $\pi _{1}(Y)=G$ such that for any integer $q\geq 0,$ there is an
isomorphism $H_{q}(W;R)\cong H_{q}(M;R)$ for the cobordism $W.$ According to
the universal coefficient theorem, for all integers $q\geq 0$ the relative
cohomology groups $H^{q}(W,M;R)=0.$ Therefore, for any integer $q\geq 0,$
there is an isomorphism $H_{q}(Y;R)\cong H_{q}(M;R)$ by Theorem \ref{th3}%
\textbf{\ }(ii).
\end{proof}

\bigskip

Recall that an $n$-dimensional $R$-homology sphere is an $n$-dimensional
manifold $Y$ such that $H_{i}(Y;R)=H_{i}(S^{n};R)$ for any integer $i\geq 0.$
The first part of the following result proved by Kervaire in \cite{kar} is a
special case of Theorem \ref{fund} when $M=S^{n}$ and $R=\mathbb{Z}$.

\begin{corollary}
\label{ker}Let $G$ be a finitely presented group and $n\geq 5$ be an
integer. Then there exists an $n$-dimensional homology sphere $Y$ with $\pi
_{1}(Y)=G$ if and only if $G$ is superperfect, i.e. $H_{1}(G;\mathbb{Z})=0$
and $H_{2}(G;\mathbb{Z})=0.$ Moreover, such manifolds can taken to be in the
same cobordism class as $S^{n}$.
\end{corollary}

Hausmann and Weinberger \cite{hw} constructed a superperfect group $G$ for
which any $4$-manifold $Y$ with $\pi _{1}(Y)=G$ satisfies $\chi (Y)>2.$ As a
consequence it follows that Theorem \ref{fund} and Corollary \ref{ker} do
not extend to dimension four.

\subsection{The fundamental groups of higher-dimensional knots}

In this subsection, we study the fundamental groups of higher-dimensional
knots. The following result proved by Kervaire \cite{ker} is a corollary of
Theorem \ref{th3}. (Recall the definition of weight $w(G)$ from the
Introduction).

\begin{corollary}
\label{knot}Given an integer $n\geq 3,$ a finitely presentable group $G$ is
isomorphic to $\pi _{1}(S^{n+2}-f(S^{n}))$ for some differential embedding $%
f:S^{n}\rightarrow S^{n+2}$ if and only if the first homology group $H_{1}(G;%
\mathbb{Z)=Z}$, the weight of $G$ is $1$ and the second homology group $%
H_{2}(G;\mathbb{Z})=0.$
\end{corollary}

\begin{proof}
Suppose for a finitely presentable group $G$, we have $H_{1}(G;\mathbb{Z)=Z}$%
, $H_{2}(G;\mathbb{Z})=0$ and the weight $w(G)=1.$ Let $X=S^{n+2},$ $\pi =1,$
the trivial group, and $\alpha :\pi \rightarrow G$ the trivial group
homomorphism. Notice that $\alpha $ induces an injection of the first
homology groups and surjection of the second homology groups. By Theorem \ref%
{th3} with $R=\mathbb{Z}$, we get a closed manifold $Y$ with $\pi _{1}(Y)=G,$
obtained from $S^{n+2}$ by attaching $1$ -handles, $2$-handles and $3$%
-handles. Suppose that $W$ is the surgery trace. By Poincar\'{e} duality,
for any integer $q\leq n+1$ the relative cohomology group $H^{q}(W,Y)=0.$
According to the universal coefficient theorem, we have that for each
integer $2\leq i\leq n+1$ there is an isomorphism%
\begin{equation*}
H_{i}(Y)\cong H_{i}(X)=0.
\end{equation*}%
Let $\gamma \in G$ be an element such that $G$ is normally generated by $%
\gamma $ and $\varphi :S^{1}\rightarrow Y$ be a differential embedding
representing $\gamma .$ Extend $\varphi $ to be an embedding $\varphi
^{\prime }:S^{1}\times D^{n+1}\rightarrow Y.$ Do surgery to $Y$ along $%
\varphi ^{\prime }$ to get a manifold $M.$ It can be easily seen that $M$ is
simply connected and for each integer $1\leq i\leq n$ the homology group $%
H_{i}(M)=0.$ Therefore, $M$ is a $(n+2)$-sphere by the solution of
higher-dimensional Poincar\'{e} conjecture (note that all manifolds are
assumed to be smooth). Let $\phi :D^{2}\times S^{n}\rightarrow M$ be the
embedding and choose $f=\phi (0,-)$ to be the embedding $S^{n}\rightarrow M.$
It can be directly checked that
\begin{eqnarray*}
\pi _{1}(M-f(S^{n})) &\cong &\pi _{1}(M-\phi (D^{2}\times S^{n})) \\
&\cong &\pi _{1}(Y-\varphi ^{\prime }(S^{1}\times D^{n+1}))\cong \pi
_{1}(Y)=G.
\end{eqnarray*}%
This finishes the "if" part.

Conversely, suppose $f:S^{n}\rightarrow S^{n+2}$ is a differential
embedding. According to Alexander duality, we have $%
H_{2}(S^{n+2}-f(S^{n}))=0 $ and $H_{1}(S^{n+2}-f(S^{n}))=H_{1}(G;\mathbb{Z}%
)=0.$ By Hopf's theorem in Lemma \ref{lem1} (with the coefficient $V=\mathbb{%
Z}$), the second homology group $H_{2}(G;\mathbb{Z})=0.$ Let $\alpha
:S^{1}\rightarrow S^{n+2}-f(S^{n}) $ be an embedding such that $\alpha
(S^{1})$ bounds a small $2$-disc in $S^{n+2}$ that intersects $f(S^{n})$
transversally at exactly one point. Then the group $\pi
_{1}(S^{n+2}-f(S^{n}))$ is normally generated by the element represented by $%
\alpha .$ For more details, see the proof of Lemma 2 in Kervaire \cite{ker}. This
proves the weight $w(G)=1$ and finishes the proof.
\end{proof}

Corollary \ref{knot} is Theorem 1 in Kervaire \cite{ker}. Similarly, we can show that
Theorem 3 in \cite{ker} concerning the fundamental groups of links is also a
corollary of Theorem \ref{th3}. That is for an integer $n\geq 3,$ a finitely
presentable group $G$ is isomorphic to $\pi _{1}(S^{n+2}-L_{k})$ for some $k$
disjointly embedded $n$-spheres $L_{k}$ if and only if $H_{1}(G;\mathbb{Z})=%
\mathbb{Z}^{k},$ $w(G)=k$ and $H_{2}(G;\mathbb{Z})=0$ (cf. Theorem 3 in \cite%
{ker}). The proof is of the same pattern as that of Corollary \ref{knot} and
will be left to the reader.

\subsection{Zero-in-the-spectrum conjecture}

In the notation of the Introduction, \emph{zero not belonging to the spectrum%
} of $\Delta =\Delta _{\ast }$ can also be expressed as the vanishing of $%
H_{\ast }(M;C_{r}^{\ast }(\pi _{1}(M))).$ The following is a version of the
zero-in-the-spectrum conjecture using homology. For more details, we refer
the reader to the book of L\"{u}ck \cite{LU}.

\begin{conjecture}
Let $M$ be a closed, connected, oriented and aspherical manifold with
fundamental group $\pi .$ Then for some $i\geq 0,$ $H_{i}(X;C_{r}^{\ast
}(\pi ))\neq 0.$
\end{conjecture}

If the condition that $X$ is aspherical is dropped, the following corollary,
which is a special case of Theorem \ref{th3} when $R=C_{\mathbb{R}}^{\ast
}(G)\ $and $\pi =1$, shows the above conjecture is not true. This result is
a generalization of the results obtained by Farber-Weinberger \cite{FW} and
Higson-Roe-Schick \cite{hig}. Recall that the real $C^{\ast }$-algebra $C_{%
\mathbb{R}}^{\ast }(G)$ is a finitely $G$-dense ring.

\begin{corollary}[\protect Higson-Roe-Schick \cite{hig}]
\label{zero}Let $G$ be a finitely presented group with the homology groups
\begin{equation*}
H_{0}(G;C_{r}^{\ast }(G))=H_{1}(G;C_{r}^{\ast }(G))=H_{2}(G;C_{r}^{\ast
}(G))=0.
\end{equation*}%
For every integer $n\geq 6$ there is a closed manifold $M$ of dimension $n$
such that $\pi _{1}(Y)=G$ and for each integer $n\geq 0,$ the homology group
$H_{n}(Y;C_{r}^{\ast }(G))=0.$
\end{corollary}

\begin{proof}
According to Proposition 4.8 in Ye \cite{Ye}, the vanishing of lower degree
homology groups with coefficients $C_{r}^{\ast }(G)$ is the same as that
with coefficients $C_{\mathbb{R}}^{\ast }(G).$ Then we have%
\begin{equation*}
H_{0}(G;C_{\mathbb{R}}^{\ast }(G))=H_{1}(G;C_{\mathbb{R}}^{\ast
}(G))=H_{2}(G;C_{\mathbb{R}}^{\ast }(G))=0.
\end{equation*}%
Note that the real $C^{\ast }$-algebra $C_{\mathbb{R}}^{\ast }(G)$ is a
finitely $G$-dense ring. Let $\alpha :\pi =1\rightarrow G,$ $R=C_{\mathbb{R}%
}^{\ast }(G)$ and $X=S^{n}$ in Theorem \ref{th3}. By the long exact sequence
of homology groups%
\begin{eqnarray*}
\cdots &\rightarrow &H_{1}(1;C_{\mathbb{R}}^{\ast }(G))\rightarrow
H_{1}(G;C_{\mathbb{R}}^{\ast }(G))\rightarrow H_{1}(G,1;C_{\mathbb{R}}^{\ast
}(G)) \\
&\rightarrow &H_{0}(1;C_{\mathbb{R}}^{\ast }(G))\rightarrow H_{0}(G;C_{%
\mathbb{R}}^{\ast }(G))\rightarrow 0,
\end{eqnarray*}%
we have
\begin{equation*}
H_{1}(G,\pi ;C_{\mathbb{R}}^{\ast }(G))=H_{0}(1;C_{\mathbb{R}}^{\ast
}(G))=C_{\mathbb{R}}^{\ast }(G),
\end{equation*}%
which is a free $C_{\mathbb{R}}^{\ast }(G)$-module. Therefore, there exists
a closed manifold $Y$ by Theorem \ref{th3} such that for any integer $0\leq
q\leq [ n/2 ],$ the homology group $H_{q}(Y;C_{\mathbb{R}}^{\ast
}(G))=0.$ According to the universal coefficients theorem and Poincar\'{e}
duality for $L^{2}$-homology (cf. Theorem 6.6 and Theorem 6.7 in Farber \cite{fab}), we get that for any integer $q\geq 0,$ the homology group $H_{n}(Y;C_{r}^{\ast }(G))=0.$
\end{proof}

\bigskip

\noindent \textbf{Acknowledgements}

The author wants to thank Professor Wolfgang L\"{u}ck for supporting from
his Leibniz-Preis a visit to Hausdorff Center of Mathematics in University
of Bonn from April 2011 to July 2011, when parts of this paper were written.
He is also grateful to his advisor Professor A. J. Berrick for many helpful
discussions.

Department of Mathematics, National University of Singapore, Kent Ridge
119076, Singapore.\\
Current address:

Mathematical Institute, University of Oxford, 24-29 St Giles', Oxford, OX1
3LB, UK. \\
Eamil: Shengkui.Ye@maths.ox.ac.uk

\end{document}